\documentclass{amsart}

\usepackage[english]{babel}
\usepackage{amsmath,amssymb,amsfonts,amsthm}
\usepackage{todonotes}
\usepackage{arydshln}
\usepackage[centering]{geometry}
\usepackage{xcolor}
\usepackage[utf8]{inputenc}

\usepackage{hyperref}
\hypersetup{colorlinks   = true,
  		    citecolor    = green,
   		    linkcolor    = blue}
\usepackage{cleveref}
\usepackage{mathtools}
\usepackage{ mathrsfs }
\usepackage{environ}
\usepackage{caption}
\usepackage{systeme}
\usepackage{subcaption}
\usepackage{listings}

\def\<{\left < }
\def\>{\right >}
\def\({\left ( }
\def\){\right )}

\newcommand{\de}{\mathrm{d}}

\DeclareMathOperator{\spanned}{span}
\DeclareMathOperator{\nor}{nor}
\DeclareMathOperator{\hor}{hor}

\DeclareMathOperator{\dd}{d}
\newtheorem{theorem}{Theorem}[section]   
\newtheorem*{theorem*}{Theorem}          
\newtheorem{lemma}[theorem]{Lemma}
\newtheorem{proposition}[theorem]{Proposition}
\newtheorem*{genericthm*}{\thistheoremname}

\theoremstyle{definition}
\newtheorem{definition}[theorem]{Definition}

\newtheorem*{remark}{Remark}
\newcommand{\thistheoremname}{}

\theoremstyle{plain}
\newtheorem*{namedthm}{\namedthmname}
\newcounter{namedthm}

\makeatletter
\newenvironment{named}[1]
  {\def\namedthmname{#1}%
   \refstepcounter{namedthm}%
   \namedthm\def\@currentlabel{#1}}
  {\endnamedthm}
\makeatother

\title[Minimal $\delta(2)$-ideal Lagrangian submanifolds and the Quaternionic projective space]{Minimal $\boldsymbol{\delta}\boldsymbol{(}\boldsymbol{2}\boldsymbol{)}$-ideal Lagrangian submanifolds and the Quaternionic projective space}
\author{Kristof Dekimpe, Joeri Van der Veken and Luc Vrancken}
\thanks{All three authors are supported by the Research Foundation--Flanders (FWO) and the National Natural Science Foundation of China (NSFC) under collaboration project G0F2319N. J. Van der Veken is also supported  
by the KU Leuven Research Fund under project 3E210539 and by the Research Foundation--Flanders (FWO) and the Fonds de la Recherche Scientifique (FNRS) under EOS Projects G0H4518N and G0I2222N}
\address{{(Kristof Dekimpe, Joeri Van der Veken, Luc Vrancken)\;}KU Leuven, Department of Mathematics, Celestijnenlaan 200B - Box 2400, 3001 Leuven, Belgium}
\email{kristof.dekimpe@kuleuven.be}
\email{joeri.vanderveken@kuleuven.be}
\email{luc.vrancken@kuleuven.be}
\address{{(Luc Vrancken)\;}{LMI-Laboratoire de Mathématiques pour l’Ingénieur, Université Polytechnique Hauts-de-France, Campus du Mont Houy, 59313 Valenciennes Cedex 9, France}}
\email{luc.vrancken@univ-valenciennes.fr}

\begin{document}

\maketitle
\begin{abstract}We construct an explicit map from a generic minimal $\delta(2)$-ideal Lagrangian submanifold of $\mathbb{C}^n$ to the quaternionic projective space $\mathbb{H}P^{n-1}$, whose image is either a point or a minimal totally complex surface. A stronger result is obtained for $n=3$, since the above mentioned map then provides a one-to-one correspondence between minimal $\delta(2)$-ideal Lagrangian submanifolds of $\mathbb{C}^3$ and minimal totally complex surfaces in $\mathbb{H}P^2$ which are moreover anti-symmetric. Finally, we also show that there is a one-to-one correspondence between such surfaces in $\mathbb{H}P^2$ and minimal Lagrangian surfaces in $\mathbb{C}P^2$.

\end{abstract}
\section*{Introduction}
One of the most fundamental problems in the theory of submanifolds is the question of immersibility of a Riemannian manifold in a Euclidean space. The famous 1956 embedding theorem of Nash \cite{Nash} states that every Riemannian manifold can be isometrically embedded into a Euclidean space of sufficiently high dimension. This theorem presented the hope that seeing any Riemannian manifold as a Riemannian submanifold of a bigger Euclidean space would enable the use of extrinsic geometry to study the (intrinsic) geometry of the Riemannian manifold. Unfortunately, this approach had not yet yielded the expected results in 1985 according to Gromov \cite{Gromov}. This was mostly due to the lack of control over the extrinsic geometry of the submanifolds in terms of the known intrinsic geometric invariants. This was the main motivation for Chen to introduce a new type of curvature invariants, the so-called $\delta$-invariants, together with inequalities involving these new intrinsic invariants and the mean curvature vector of an isometric immersion. For a detailed overview of the study on Chen's $\delta$-invariants we refer to his book \cite{book}. 

The aforementioned inequalities become very interesting in the case of a Lagrangian submanifold $M^n$ immersed in a complex space form $\tilde{M}^n(4c)$, see \cite{FinalSolution}. Recall that an isometric immersion $~{M^n\rightarrow \tilde{M}^n(4c)}$ is Lagrangian if the complex structure $J$ of $\tilde{M}^n(4c)$ is an isomorphism between the tangent space and the normal space of the submanifold $M^n$ at any of its points. In this paper the focus will be on the first $\delta$-invariant, nowadays denoted by $\delta(2)$. A Lagrangian submanifold which attains equality in the inequality at all of its points is called a \textit{$\delta(2)$-ideal} Lagrangian submanifold. An overview of ideal Lagrangian submanifolds can be found in \cite{ideal}.

An important concept in the study of d(2)-ideal submanifolds is the distribution
\begin{align}\label{eq:IntroD1}
\mathcal{D}_1=\spanned\left\{ E_1, E_2\right\},
\end{align}
where $E_1$ and $E_2$ are orthonormal vector fields spanning the plane of minimal sectional curvature at every point. Classification results can be obtained under integrability conditions of this distribution, see for example \cite{ideal} for the Lagrangian case. 

Another class of submanifolds, which will be featured in this paper, are the totally complex submanifolds of the quaternionic projective space $\mathbb{H}P^{n-1}$. Recall that $\mathbb{R}^{4n}\cong \mathbb{H}^{n}$ has a standard quaternionic structure $\left\{i, j, k \right\}$ and that $\mathbb{H}P^{n-1}$ gains a local basis $\left\{I, J, K\right\}$ of the quaternionic vector bundle $V$ through the Hopf fibration $\pi: S^{4n-1}\rightarrow\mathbb{H}P^{n-1}$. We call a triple $\left\{J_1, J_2, J_3\right\}$ an adapted basis of a submanifold $N$ of $\mathbb{H}P^{n-1}$ if it forms a local basis for the quaternionic vector bundle $V$, when restricted to the tangent space $TN$. A submanifold $(N, J_1)$ of $\mathbb{H}P^{n-1}$, with adapted basis $\left\{J_1, J_2, J_3\right\}$, is then called a \textit{totally complex submanifold} if the tangent space $TN$ is invariant under $J_1$ and $J_2TN$ is orthogonal to $TN$.

The aim of this paper is to find a correspondence between $\delta(2)$-ideal submanifolds of the complex space form $\mathbb{C}^n$ and totally complex submanifolds of the quaternionic projective space $\mathbb{H}P^{n-1}$. A specific map is constructed from a generic minimal $\delta(2)$-ideal Lagrangian submanifold $M^n$ of $\mathbb{C}^n$ to $\mathbb{H}P^{n-1}$, by considering the special vector fields $E_1$ and $E_2$ spanning the distribution $\mathcal{D}_1$ on $M^n$. It is then shown that this map is either a constant map or that its image is a minimal totally complex surface in $\mathbb{H}P^{n-1}$, depending on whether $\mathcal{D}_1$ is totally geodesic or not. 
\begin{named}{Main Theorem 1}\label{MT1}
Let $M^n\subset \mathbb{C}^n$ be a minimal $\delta(2)$-ideal Lagrangian submanifold. Consider the map $\psi:M^n\rightarrow \mathbb{H}P^{n-1}:p\mapsto[E_1(p)+iE_2(p)]$, where $i$ is the canonical complex structure on $\mathbb{C}^{2n}$ and $E_1$ and $E_2$ are the vector fields spanning distribution $\mathcal{D}_1$ as in (\ref{eq:IntroD1}).
\begin{itemize}
\item[(1)] If the distribution $\mathcal{D}_1$ is totally geodesic, then $\psi$ is a constant map.
\item[(2)] Otherwise, the image of $\psi$ is a minimal totally complex surface in $\mathbb{H}P^{n-1}$. 
\end{itemize}
\end{named}

The other main results of this paper are obtained when the dimension $n$ equals 3. We introduce a new class of totally complex submanifolds, the \textit{anti-symmetric totally complex} submanifolds. These are totally complex submanifolds $(N, J_1)$, with an adapted basis $\left\{J_1, J_2, J_3 \right\}$, satisfying an additional condition on their second fundamental form $h$, namely that the maps $(X,Y, Z)\mapsto g( h(X,Y),J_2 Z )$ and $(X, Y, Z)\mapsto g( h(X,Y),J_3 Z )$ are anti-symmetric in the last two indices, with $g$ the metric tensor on $\mathbb{H}P^2$ and $X,Y,Z$ vector fields on $N^2$. 

The following theorems then provide a correspondence between these surfaces in $\mathbb{H}P^2$ and 3-dimensional minimal $\delta(2)$-ideal Lagrangian submanifolds in $\mathbb{C}^3$. 
\begin{named}{Main Theorem 2A}\label{MT2A}
Let $M^3\subset \mathbb{C}^3$ be a minimal $\delta(2)$-ideal Lagrangian submanifold, with an integrable but not totally geodesic distribution $\mathcal{D}_1$ as in (\ref{eq:IntroD1}). The image of $M^3$ under the map $\psi$, defined in \ref{MT1}, is a minimal totally complex surface in $\mathbb{H}P^2$, which is moreover anti-symmetric. Conversely, for every minimal anti-symmetric totally complex surface in $\mathbb{H}P^2$, there exists a 3-dimensional minimal $\delta(2)$-ideal Lagrangian submanifold $M^3\subset\mathbb{C}^3$, with an integrable but not totally geodesic distribution $\mathcal{D}_1$, such that $\psi(M^3)$ is the given surface.
\end{named}
An analogous result is also obtained when the distribution $\mathcal{D}_1$ is nowhere integrable.
\begin{named}{Main Theorem 2B}\label{MT2B}
Let $M^3\subset \mathbb{C}^3$ be a minimal $\delta(2)$-ideal Lagrangian submanifold, with nowhere integrable distribution $\mathcal{D}_1$ as in (\ref{eq:IntroD1}). The image of $M^3$ under the map $\psi$, defined in \ref{MT1}, is a minimal totally complex surface in $\mathbb{H}P^2$, which is moreover anti-symmetric. Conversely, for every minimal anti-symmetric totally complex surface in $\mathbb{H}P^2$ and every eigenfunction of the Laplace-Beltrami operator with eigenvalue $-4$, there exists a 3-dimensional minimal $\delta(2)$-ideal Lagrangian submanifold $M^3\subset\mathbb{C}^3$, with a nowhere integrable distribution $\mathcal{D}_1$, such that $\psi(M^3)$ is the given surface.
\end{named}

This paper is structured as follows: we start with some preliminaries, focussing on defining $\delta(2)$-ideal submanifolds, quaternionic K\"{a}hler manifolds and their relevant submanifolds. In Section \ref{section:MinimalQuaternionic} a proof is given for \ref{MT1}. The map $\psi$ between minimal $\delta(2)$-ideal Lagrangian submanifolds of $\mathbb{C}^n $ and the quaternionic projective space $\mathbb{H}P^{n-1}$ is constructed and it is shown that this is either a constant map or defines a minimal totally complex surface in $\mathbb{H}P^{n-1}$. In Section \ref{section:Dim3} the special case of dimension $n=3$ is considered. It is shown that $\psi$ maps a $3$-dimensional minimal $\delta(2)$-ideal Lagrangian submanifold to a minimal anti-symmetric totally complex surface in $\mathbb{H}P^2$. By completely characterizing the minimal anti-symmetric totally complex surfaces in $\mathbb{H}P^2$, we can then present the proof of \ref{MT2A} and of \ref{MT2B}. Section \ref{Section:reformulation} then finally gives a one-to-one correspondence between these surfaces in $\mathbb{H}P^2$ and minimal Lagrangian surfaces in the complex projective plane $\mathbb{C}P^2$, allowing us to reformulate \ref{MT2A} and \ref{MT2B}, cfr. \ref{MT2A'} and \ref{MT2B'}.
\section{Preliminaries}
In this section we recall some basic definitions, properties and formulas regarding Riemannian submanifolds, $\delta(2)$-ideal Lagrangian submanifolds and quaternionic K\"{a}hler manifolds. More specifically, we will elaborate on the construction of the quaternionic projective space $\mathbb{H}P^n$ as a quaternionic space form.
\subsection{Riemannian submanifolds}
Let $M$ be a Riemannian submanifold of an ambient Riemannian manifold $(\tilde{M},g)$. The following formulas and notations will be extensively used throughout this paper. Denote by $\nabla$ and $\tilde{\nabla}$ the Levi-Civita connections on $M$ and $\tilde{M}$, respectively.  The formulas of Gauss and Weingarten respectively state that
\begin{align}
& \widetilde{\nabla}_X Y=\nabla_X Y + h(X,Y), \label{Gaussformula}\\
& \widetilde{\nabla}_X \xi = -A_{\xi}X + \nabla^{\perp}_X \xi\label{Weingarten}
\end{align}
for vector fields $X$ and $Y$ tangent to $M$ and a vector field $\xi$ normal to $M$. Here, $h$ is the second fundamental form, taking values in the normal bundle, $A_{\xi}$ is the shape operator associated to the normal vector field $\xi$ and $\nabla^{\perp}$ is the normal connection. The mean curvature vector field $H$ is defined as the averaged trace of the second fundamental form $h$. A submanifold $M$ is a minimal submanifold if the mean curvature vector field $H$ is everywhere zero.
The second fundamental form and the shape operator are related by
\begin{align*}
g(h(X,Y),\xi)=g(A_\xi X,Y),
\end{align*}
while the normal connection is used in defining the covariant derivative of the second fundamental form as
\begin{align*}
(\overline{\nabla}h)(X,Y,Z)=\nabla^\perp_X h(Y,Z)-h(\nabla_XY)-h(Y,\nabla_x Z),
\end{align*}
for all vector fields $X$, $Y$ and $Z$ tangent to $M$.
Denote by $X^T$ and $X^\perp$ the tangent and normal part of a vector field $X$ with respect to the submanifold $M$. The equations of Gauss, Codazzi and Ricci will also be used in this paper and can now respectively be stated as 
\begin{align}
(\tilde{R}(X,Y)Z)^T&=R(X,Y)Z +A_{h(X,Z)}Y- A_{h(Y,Z)}X \label{Gaussequation},\\
(\tilde{R}(X,Y)Z)^\perp&=(\overline{\nabla}h)(X,Y,Z)-(\overline{\nabla}h)(Y,X,Z)\label{Codazzi},\\
(\tilde{R}(X,Y)\xi)^\perp&=R^\perp(X,Y)\xi+h(A_\xi X,Y)-h(X, A_\xi Y),\label{Ricci}
\end{align}
for tangent vector fields $X,Y,Z$ and normal vector field $\xi$. Here, $R$ and $\tilde{R}$ are the curvature tensors of $M$ and $\tilde{M}$, respectively, while $R^\perp$ is the normal curvature tensor. 
\subsection{$\mathbf{\boldsymbol{\delta}(2)}$-ideal Lagrangian immersions}
It is possible to define a Riemannian invariant, denoted by $\delta(2)$, for an arbitrary Riemannian manifold $M^n$ of dimension $n\geq 3$, as detailed in \cite{Obstruction}. This invariant is defined at each point $p\in M^n$ by
\begin{align*}
\delta(2)(p)=\tau(p)-\inf\left\{K(\pi)\;\vert\; \pi\subset T_p M \text{ is a plane} \right\},
\end{align*}
where $K$ denotes the sectional curvature of $M^n$ and $\tau$ is the scalar curvature of $M^n$, defined as
\begin{align*}
\tau(p)=\sum_{i<j}K(e_i\wedge e_j),
\end{align*}
with $\left\{e_1,\ldots,  e_n \right\}$ an orthonormal basis of $T_p M^n$.
Note that the $\delta(2)$ curvature is just one in a sequence of invariants constructed by B. Y. Chen \cite{book}.
For Lagrangian submanifolds $M^n$ of $n$-dimensional complex space forms $\tilde{M}^n(4c)$ with constant holomorphic sectional curvature $4c$, the following inequality regarding the $\delta(2)$ curvature invariant of the submanifold $M^n$ was proven in \cite{FinalSolution}.
\begin{theorem}
Let $M^n\subset \tilde{M}^n(4c)$, with $n\geq3$, be a Lagrangian submanifold. Then the following inequality holds at every point of $M^n$:
\begin{align}\label{eq:ineqdelta}
\delta(2)\leq\frac{n^2(2n-3)}{2(2n+3)}\left \| H \right \|^2+\frac{1}{2}(n+1)(n-2)c,
\end{align}
with $H$ the mean curvature vector field of $M^n$. 
\end{theorem}
This is thus an inequality pertaining intrinsic invariant of the submanifold, the $\delta(2)$ curvature invariant of the Riemannian manifold $M^n$, while also containing extrinsic invariants, thus also depending on the specific immersion of $M^n$ in the ambient space $\tilde{M}^n(4c)$. Note that an immediate consequence of the previous theorem is that a necessary condition for $M^n$ to allow a minimal Lagrangian immersion into $\mathbb{C}^n$ is that $\delta(2)$ has to be non-positive at every point of $M^n$. If the previous inequality (\ref{eq:ineqdelta}) becomes an equality for every $p\in M^n$, then $M^n$ is called a $\delta(2)$-ideal Lagrangian submanifold of $\tilde{M}^n(4c)$.


At every point $p$ of a Lagrangian submanifold $M^n$ we define the kernel of the second fundamental form $h$ by
\begin{align}\label{eq:kern}
\mathcal{D}(p)=\left\{X\in T_p M\;\vert\;\forall Y\in T_p M:h(X,Y)=0 \right\}.
\end{align}
It is then possible to define a local frame field on $M^n$ as stated by the following lemma in \cite{Delta2}. 
\begin{lemma}\label{lem:2ndfund}{\cite{Delta2}}
Let $M^n$ be a minimal Lagrangian $\delta(2)$-ideal submanifold of a complex space form $\tilde{M^n}(4c)$ with complex structure $J$. Assume that the dimension of the distribution $\mathcal{D}$ is constantly equal to $n-2$ and let $p\in M^n$. Then, there exists local orthonormal vector fields $E_1, \ldots, E_n$ on a neighbourhood of $p$ such that for all $i\geq1$ and $j\geq3$
\begin{alignat*}{3}
&h(E_1, E_1)&&=\lambda JE_1, \qquad{} &&h(E_1, E_2)=-\lambda JE_2,\\
&h(E_2, E_2)&&=-\lambda JE_1,\qquad{} &&h(E_i, E_j)=0,
\end{alignat*}
where $\lambda$ is a positive function determined by $4\lambda^2=n(n-1)c-\tau$ and $\tau$ is the scalar curvature of $M^n$.
\end{lemma}
\begin{remark}
Note that, if the dimension of the distribution $\mathcal{D}$ constant, then it is either everywhere equal to $n$, making $M^n$ totally geodesic, or it is everywhere equal to $n-2$. As it is always possible to define this frame $\left\{E_1, \ldots, E_n\right\}$ in a point $p$, whenever $M^n$ is not totally geodesic, the assumption of constant dimension of $\mathcal{D}$ in the above is just to assure the smoothness of the frame.
\end{remark}
Throughout this paper we will always assume that the dimension of $\mathcal{D}$ is always equal to $n-2$.  The vector fields $E_1$ and $E_2$ also have a geometrical interpretation, since at every point $p$ the sectional curvature of the submanifold attains a minimum on the plane spanned by $E_1(p)$ and $E_2(p)$.
 
Let $\gamma_{ij}^k=g(\nabla_{E_i}E_j, E_k)$ be the connection coefficients related to the local frame $\left\{E_1,\ldots, E_n \right\}$ as defined by the previous lemma. It is straightforward to check that $\gamma_{ij}^k=-\gamma_{ik}^j$. The following lemma then yields other relations for the connection coefficients. 
\begin{lemma}\label{lem:connect}{\cite{Delta2}}
Under the assumptions of Lemma \ref{lem:2ndfund}, the local functions $\gamma_{ij}^k=g(\nabla_{E_i}E_j, E_k)$ satisfy the following relations for $i,j>2$:
\begin{align}
&\gamma_{11}^{i}-\gamma_{22}^{i}=0,\\
&\gamma_{12}^{i}+\gamma_{21}^{i}=0,\\
&\gamma_{ij}^{1}=\gamma_{ij}^{2}=0,\\
&\gamma_{i1}^{2}=-\frac{1}{3}\gamma_{12}^{i}.
\end{align}
Moreover, the function $\lambda$ satisfies the following system of differential equations for $i>2$:
\begin{align*}
&E_1(\lambda)=-3\lambda\gamma_{21}^2,\\
&E_2(\lambda)=3\lambda\gamma_{11}^2,\\
&E_i(\lambda)=-\lambda\gamma_{1i}^1.
\end{align*}
\end{lemma}
\subsection{Quaternionic K\"{a}hler manifolds}
In this paper we will consider surfaces in the quaternionic projective space $\mathbb{H}P^n$, which is an example of a quaternionic K\"{a}hler manifold. We first introduce the notion of a quaternionic K\"ahler manifold $\tilde{M}^{4n}$ and follow the notation from \cite{RealFunabashi}.  
Let $\tilde{M}^{4n}$ be a smooth manifold of dimension $4n$ and assume that $\tilde{M}^{4n}$ satisfies the following conditions.
\begin{itemize}
\item[(a)]$\tilde{M}^{4n}$ admits a 3-dimensional vector bundle $V$ consisting of tensors of type (1,1) over $\tilde{M}^{4n}$ satisfying the condition that every point $p$ in $\tilde{M}^{4n}$ has a neighbourhood $U$ with a local basis $\left\{I, J, K \right\}$ of $V$ such that
\begin{align*}
&I^2=J^2=K^2=-\mathrm{id},\\
IJ=-JI=K,\quad &IK=-KI=-J,\quad JK=-KJ=I,
\end{align*}
where $\mathrm{id}$ is the identity tensor of type (1,1) on $\tilde{M}^{4n}$. Such a triplet $\left\{ I,J,K \right\}$ is called a local basis of $V$ on $U$. 
\item[(b)]$\tilde{M}^{4n}$ admits a Riemannian metric $g$ such that, for any canonical local basis $\left\{ I,J,K \right\}$ of $V$ in $U$, the local tensor fields $I, J, K$ are almost Hermitian with respect to $g$ and the equations
\begin{align*}
\nabla_X I&=r(X)J-q(X)K,\\
\nabla_X J&=-r(X)I+p(X)K,\\
\nabla_X K&=q(X)I-p(X)J
\end{align*}
are satisfied for any vector field $X$ in $\tilde{M}^{4n}$, with $\nabla$ denoting the Levi-Civita connection determined by $g$, where $p, q$ and $r$ are 1-forms defined on $U$. 
\end{itemize}
Such a triplet ($\tilde{M}^{4n},g,V)$ is called a quaternionic K\"{a}hler manifold with quaternionic K\"{a}hler structure $(g,V)$ and is often simply denoted by $\tilde{M}^{4n}$.  At every point $p\in \tilde{M}^{4n}$ we can define for a vector field $X$ a four-dimensional subspace $Q(X)$, called the $Q$-section, as
\begin{align*}
Q(X)=\left\{Y\;\vert\;Y=aX+bIX+cJX+dKX \right\},
\end{align*}
with $a,b,c,d\in\mathbb{R}$. If for any $Y, Z\in Q(X)$ the sectional curvature $K(Y,Z)$ is a constant $\rho(X)$, then we call it the $Q$-sectional curvature of  $\tilde{M}^{4n}$ with respect to the vector field $X$ and the point $p$. If a quaternionic K\"{a}hler manifold $\tilde{M}^{4n}$ has constant $Q$-sectional curvature $c$ at every point, then it is called a quaternionic space form and denoted by $\tilde{M}^{4n}(4c)$. The curvature operator $R$ of a quaternionic space-form $\tilde{M}^{4n}(4c)$ has the form
\begin{align}\label{def:curvature}
R(X,Y)Z=\frac{c}{4}\left[\Delta(Y,Z)X-\Delta(X,Z)Y-2\Gamma(X,Y)Z \right],
\end{align}
where 
\begin{align*}
\Delta(Y,Z)X&=g(Y,Z)X+g(IY, Z)IX+g(JY, Z)JX+g(KY,Z)KX,\\
\Gamma(X,Y)Z&=g(IX,Y)IZ+g(JX, T)JZ+g(KX,Y)KZ.
\end{align*}
The notion of almost complex and totally complex submanifolds can be found in \cite{Alekseevsky} and will be heavily used in this paper. 
\begin{definition}\label{def:almostcomplex}
A submanifold $M^{2m}$ of a quaternionic K\"ahler manifold $(\tilde{M}^{4n},g,V)$, together with a section $J_1=J_1^M\in\Gamma(V\vert_M)$ such that 
\begin{itemize}
\item[(1)]$J_1^2=-Id$,
\item[(2)]$J_1TM=TM$,
\end{itemize}
is called \textit{almost complex}. Such a submanifold is called \textit{complex} if $J_1$ is an integrable complex structure on $M$. 
\end{definition}
One can show that every 2-dimensional almost complex submanifold $M^2$ of $\tilde{M}^{4n}$ is automatically a complex submanifold. A local basis $(J_1, J_2, J_3)$ of $V$, defined on a neighbourhood $U$ in $\tilde{M}^{4n}$ of a point $p\in M^{2m}$, such that $(M^{2m}, J_1$) is an almost complex submanifold is called an adapted basis. This leads us to the definition of a totally complex submanifold. 
\begin{definition}\label{def:totalcomplex}
An almost complex submanifold $(M^{2m},J_1)$ of  $(\tilde{M}^{4n},g,V)$ with an adapted basis $(J_1, J_2,J_3)$ is called \textit{totally complex} if $J_2 TM^{2m}\perp TM^{2m}$.
\end{definition}
\begin{remark}
The definition for a totally complex submanifold does not imply that it is K\"ahler, i.e. that $\tilde{\nabla}J_1$ is zero on the submanifold, as opposed to the definition given in \cite{ComplexFunabashi}
\end{remark}
In this paper we will consider totally complex surfaces $M^2$ with an additional condition on their second fundamental form and these surfaces will be called anti-symmetric totally complex surfaces.
\begin{definition}\label{def:anti-symm}
A totally complex submanifold $(M^{2m},J_1)$ of a quaternionic K\"ahler manifold $(\tilde{M}^{4n},g,V)$ with an adapted basis $(J_1, J_2,J_3)$ and second fundamental form $h$ is called \textit{anti-symmetric} if the maps $(X,Y, Z)\mapsto g( h(X,Y),J_2 Z )$ and $(X, Y, Z)\mapsto g( h(X,Y),J_3 Z )$ are anti-symmetric in the last two indices, with $X,Y,Z$ vector fields on $M^{2m}$
\end{definition}
\subsubsection{Quaternionic projective space}\label{subs:quatproj} We now introduce the the quaternionic projective space $\mathbb{H}P^n$ as an example of a quaternionic K\"{a}hler manifold, as shown in \cite{ComplexFunabashi}. Let $S^{4n+3}$ be the unit sphere in $\mathbb{R}^{4n+4}$ and denote by $\left\{i, j, k \right\}$ the standard quaternionic structure in $\mathbb{R}^{4n+4}$. Consider the Hopf fibration $\pi:S^{4n+3}\rightarrow \mathbb{H}P^n$ over the quaternionic projective space. This is a Riemannian submersion and gives rise to a quaternionic K\"{a}hler structure of $\mathbb{H}P^n$, with a local canonical basis $\left\{I,J,K \right\}$ such that for every vector field $X$ on $\mathbb{H}P^{n}$ one has
\begin{align}
IX=(\de\pi) (\nabla_{\bar{X}}iN),\label{eq:IJK}\;\;
JX=(\de\pi) (\nabla_{\bar{X}}jN),\;\;
KX=(\de\pi) (\nabla_{\bar{X}}kN).
\end{align} 
Here $N$ is the vector field such that $N_p$  is the outer normal vector of $S^{4n+3}$ at each point $p\in S^{4n+3}$ and $\nabla$ is the Levi-Civita connection of $S^{4n+3}$. The usual notation of $\bar{X}$ denoting the unique horizontal lift of $X$ in $S^{4n+3}$ is also used. Note that this construction only gives rise to a local canonical basis and not to a global basis of the quaternionic structure. It is known that $\mathbb{H}P^n$ is a quaternionic space form with constant $Q$-sectional curvature 4 \cite{RealFunabashi}.
\section{Proof of Main Theorem 1}\label{section:MinimalQuaternionic}
In this section we will give a proof of \ref{MT1}, by considering a specific map from a minimal $\delta(2)$-ideal Lagrangian submanifold of the complex space form $\mathbb{C}^n$ to the quaternionic projective space $\mathbb{H}P^{n-1}$, using the local frame introduced in Lemma \ref{lem:connect}. We will show that the image of this map is either a point or a minimal totally complex surface. This map will also be used in Section \ref{section:Dim3}, where the we will focus on the case where the dimension $n$ equals 3 to prove \ref{MT2A} and \ref{MT2B}.


Consider a minimal $\delta(2)$-ideal Lagrangian submanifold $M^n\subset \mathbb{C}^n$ and denote by $j$ the usual complex structure on $\mathbb{C}^n$. Recall that we will always assume that the distribution $\mathcal{D}$, as defined in equation (\ref{eq:kern}), is constantly equal to $n-2$. Then Lemma \ref{lem:2ndfund} implies the existence of the local frame $\left\{E_1,\ldots, E_n \right\}$ on $M^n$. We define the following map $x$ from $M^n$ to the unit sphere $S^{4n-1}$,
\begin{align*}
x:M^n\rightarrow S^{4n-1}\subseteq \mathbb{C}^{2n}:p\mapsto \frac{1}{\sqrt{2}}(E_1(p)+iE_2(p)),
\end{align*}
with $i$ the usual complex structure on $\mathbb{C}^{2n}$. By extending $j$ to $\mathbb{C}^{2n}$ by defining $ji=-ij$, we obtain a quaternionic structure $\left\{i,j,k=ij \right\}$ on $\mathbb{C}^{2n}\cong \mathbb{R}^{4n}$. It is then possible, using the Hopf fibration $\pi:S^{4n-1}\rightarrow \mathbb{H}P^{n-1}$, to construct a map
\begin{align*}
\psi=\pi\circ x:M^n\rightarrow \mathbb{H}P^{n-1}:p\mapsto\left[E_1(p)+iE_2(p) \right]
\end{align*} from $M^n$ to the quaternionic projective space $\mathbb{H}P^{n-1}$. We define two distributions $\mathcal{D}_1$ and $\mathcal{D}_2$ on $M^n$ as 
\begin{align}
\mathcal{D}_1&=\spanned\left\{E_1, E_2 \right\},\\
\mathcal{D}_2&=\spanned\left\{E_3,\ldots, E_n \right\},\label{eq:D2}
\end{align}
where $\mathcal{D}_1$ is once again spanned by the vector fields $E_1$ and $E_2$ as in equation (\ref{eq:IntroD1}).

We now present the proof of \ref{MT1}.
\begin{proof}
We first prove that $(\de\psi)(E_k)=0$ for the vector fields $E_k$ in the distribution $\mathcal{D}_2$. Let $D$ denote the Levi-Civita connection on $\mathbb{C}^{2n}$. Since
\begin{align*}
D_{E_k}\frac{1}{\sqrt{2}}(E_1+iE_2)&=\frac{1}{\sqrt{2}}(D_{E_k}E_1+iD_{E_k}E_2)\\
&=\frac{1}{\sqrt{2}}(\nabla_{E_k}E_1+h(E_k, E_1)+i\nabla_{E_k}E_2+ih(E_k, E_2))\\
&=-\frac{1}{3\sqrt{2}}\gamma_{12}^k E_2+\frac{1}{3}\gamma_{12}^k iE_1\\
&=\frac{\gamma_{12}^k}{3\sqrt{2}}i(E_1+iE_2),
\end{align*}
we find indeed that $(\de\pi)(D_{E_k}\frac{1}{\sqrt{2}}(E_1+iE_2))=0$, which implies that $(\de\psi)(E_k)=0$ for $k\in\left\{3, \ldots, n\right\}$. Moreover, 
\begin{align*}
D_{E_1}\frac{1}{\sqrt{2}}(E_1+iE_2)&=\frac{1}{\sqrt{2}}(\nabla_{E_1}E_1+h(E_1, E_1)+i\nabla_{E_1}E_2+ih(E_1, E_2))\\
&=\frac{1}{\sqrt{2}}\left(\sum_k (\gamma_{11}^k E_k+\gamma_{12}^k iE_k)+\lambda j(E_1+iE_2)\right)\\
&=\frac{1}{\sqrt{2}}\left(\sum_{k>2} (\gamma_{11}^k E_k+\gamma_{12}^k iE_k)+\gamma_{12}^1 i(E_1+iE_2)+\lambda j(E_1+iE_2)\right)
\end{align*}
and
\begin{align*}
D_{E_2}\frac{1}{\sqrt{2}}(E_1+iE_2)=\frac{1}{\sqrt{2}}\left(\sum_{k>2} (\gamma_{21}^k E_k+\gamma_{22}^k iE_k)+\gamma_{22}^1 i(E_1+iE_2)-\lambda ij(E_1+iE_2)\right).
\end{align*}
We define the following horizontal vector fields with respect to the Hopf fibration $\pi$: 
\begin{align}
\chi_1&=\hor D_{E_1}\frac{1}{\sqrt{2}}(E_1+iE_2)=\frac{1}{\sqrt{2}}\sum_{k>2} (\gamma_{11}^k+i\gamma_{12}^k )E_k\label{eq:hor1},\\
\chi_2&=\hor D_{E_2}\frac{1}{\sqrt{2}}(E_1+iE_2)=\frac{1}{\sqrt{2}}\sum_{k>2} (\gamma_{21}^k +i\gamma_{22}^k )E_k\label{eq:hor2}.
\end{align}
Suppose now that the distribution $\mathcal{D}_1$, defined in equation (\ref{eq:IntroD1}), is totally geodesic. This means that the connection coefficients $\gamma_{ij}^k$ satisfy
\begin{align*}
\gamma_{11}^{i}=\gamma_{22}^{i}=&\gamma_{12}^{i}=\gamma_{21}^{i}=0,\quad i>2.
\end{align*}
Then the horizontal vector fields $\chi_1$ and $\chi_2$ become zero, reducing $\psi$ to a constant map in $\mathbb{H}P^{n-1}$. 

If $\mathcal{D}_1$ is not totally geodesic, then $\psi(M^n)$ becomes a surface in $\mathbb{H}P^{n-1}$ with tangent vector fields
\begin{align*}
\psi_1&=(\de\pi)(\chi_1),\\
\psi_2&=(\de\pi)(\chi_2).
\end{align*}
We will now prove that this surface is totally complex. Recall that the quaternionic projective space $\mathbb{H}P^{n-1}$ has a local quaternionic structure $\left\{I,J,K \right\}$ satisfying equations (\ref{eq:IJK}).
 Using lemma (\ref{lem:connect}) we immediately find that $\chi_2=i\chi_1$. Applying the Hopf fibration then yields that $\psi_2=I\psi_1$, which shows that $\psi(M^n)$ is an almost complex surface according to Definition \ref{def:almostcomplex}. 

To show that the surface is totally complex, one can restrict $\left\{I,J,K \right\}$ to the tangent space of the almost complex surface $\psi(M^n)$ to obtain an adapted basis. Since $j\chi_1$ and $k\chi_2$ are orthogonal to both $\chi_1$ and $\chi_2$, one also has that $J(T\psi)$ is orthogonal to $T\psi$. Thus, using Definition \ref{def:totalcomplex}, one has that $\psi(M^n)$ is a totally complex surface.

The last thing to prove is that the surface $\psi(M^n)$ is minimal, meaning that the mean curvature $H$ is identically zero. Recall that Definition \ref{def:almostcomplex} of an almost complex surface did not include the K\"ahler condition, which makes minimality a non-trivial property of the surface. As the map $\psi$ is the composition of the Hopf fibration $\pi$ with the map $x:M^n\rightarrow S^{4n-1}$, one can considering the vector field $\mathcal{\xi}$ on $S^{4n-1}$, defined as
\begin{align*}
\xi=\hor(\hor D_{E_1}\chi_1+\hor D_{E_2}\chi_2),
\end{align*}
where $\hor$ denotes the horizontal part of the vector field with respect to the Hopf fibration. We will now show that this vector $\xi$ has no parts that are mutually orthogonal to the horizontal vector fields $\chi_1$ and $\chi_2$. This will then imply that that the mean curvature vector field $H$ of the surface $\psi(M^n)\subset\mathbb{H}P^{n-1}$, which is exactly $\dd \pi(\xi)$, will have no mutually orthogonal parts to the tangent vector fields $\psi_1=\dd \pi(\chi_1)$ and $\psi_1=\dd \pi(\chi_1)$, thus making it zero everywhere on $\psi(M^n)$. 
Remark that $\chi_1$ and $\chi_2$ can be written as
\begin{align*}
\chi_1&=(\nabla_{E_1}E_1)_{2}+ i(\nabla_{E_1}E_2)_{2},\\
\chi_2&=(\nabla_{E_2}E_1)_{2}+ i(\nabla_{E_2}E_2)_{2},
\end{align*}
where the subscript denotes that we only consider the components of the expression lying in $\mathcal{D}_2$, defined in equation (\ref{eq:D2}). One can show that $~{(\nabla_{E_1}E_1)_{2}=(\nabla_{E_2}E_2)_{2}}$ and $~{(\nabla_{E_1}E_2)_{2}=-(\nabla_{E_2}E_1)_{2}}$, by applying \ref{lem:connect}.
We can thus rewrite the vector field $\xi$ as
\begin{align*}
\xi=\hor(\nabla_{E_1}(\nabla_{E_2}E_2)_2-i\nabla_{E_1}(\nabla_{E_2}E_1)_2-\nabla_{E_2}(\nabla_{E_1}E_2)_2+i\nabla_{E_2}(\nabla_{E_1}E_1)_2)
\end{align*}

Using the definition of the Riemannian curvature tensor it is possible to rewrite the previous expression as
\begin{align}\label{eq:mincurv}
\xi=\hor(R(E_1,E_2)E_2+\nabla_{[E_1,E_2]}E_2+iR(E_2, E_1)E_1+i\nabla_{[E_2,E_1]}E_1).
\end{align}
We first consider the term $\nabla_{[E_1,E_2]}E_2+i\nabla_{[E_2,E_1]}E_1$, writing
\begin{align*}
\nabla_{[E_1,E_2]}E_2+i\nabla_{[E_2,E_1]}E_1&=D_{[E_1,E_2]}E_2-h([E_1,E_2],E_2)-iD_{[E_1,E_2]}E_1+ih([E_1,E_2],E_2),\\
&=-iD_{[E_1,E_2]}(E_1+iE_2)-h([E_1,E_2],E_2)+ih([E_1,E_2],E_1).\nonumber
\end{align*}
It is possible, using analogous calculations as in the beginning of this proof, that $iD_{[E_1,E_2]}(E_1+iE_2)$ only consists of vertical components with respect to the Hopf fibration $\pi$ and of component lying in the direction of $\chi_1$ and $\chi_2$. Thus it will not contribute to the vector field $\xi$ having components that are  orthogonal to both $\chi_1$ and $\chi_2$. Applying Lemma \ref{lem:2ndfund} to the other terms of the previous equations gives the following result:
\begin{align*}
ih([E_1,E_2],E_1)-h([E_1,E_2],E_2)&=\lambda k(\gamma_{12}^1 E_1-\gamma_{21}^2 E_2-\lambda j(\gamma_{21}^2 E_1-\gamma_{12}^1E_2)\\
&=\lambda\gamma_{12}^1 k(E_1+iE_2)-\lambda \gamma_{21}^2j(E_1+iE_2),
\end{align*}
which clearly only contains vertical components with respect to $\pi$. Thus the only terms that are left in equation (\ref{eq:mincurv}) are $R(E_1,E_2)E_2$ and $iR(E_2,E_1)E_1$. Applying the Gauss equation (\ref{Gaussequation}) yields
\begin{align*}
R(E_1,E_2)E_2&=(\tilde{R}(E_1,E_2)E_2)^{T}+A_{h(E_2,E_2)}E_1-A_{h(E_1,E_2)}E_2,\\
iR(E_2,E_1)E_1&=i(\tilde{R}(E_2,E_1)E_1)^{T}+iA_{h(E_1,E_1)}E_2-iA_{h(E_2,E_1)}E_1,
\end{align*}
where $\tilde{R}^T$ denotes the tangent part of the curvature tensor with respect to the map $x:M^3\rightarrow S^{4n-1}$  of $\mathbb{H}P^{n-1}$. Through the definition of the curvature tensor of the quaternionic space $\mathbb{H}P^{n-1}$, shown in equation (\ref{def:curvature}), one can see that $(\tilde{R}(E_1,E_2)E_2)^T$ and $i(\tilde{R}(E_2,E_1)E_1)^T$ will have no horizontal components with respect to the Hopf fibration. Thus their projections will have no normal parts to the surface $\psi(M^3)\subset \mathbb{H}P^{n-1}$. Using Lemma \ref{lem:2ndfund} then shows that
\begin{align*}
R(E_1,E_2)E_2+iR(E_2,E_1)E_1&=-2\lambda^2 E_1-2\lambda^2 iE_2\\&=-\lambda^2(E_1+iE_2),
\end{align*}
which has no horizontal part. Thus we have proven that $\dd\pi(\xi)$ has no mutually orthogonal parts to the tangent vector fields $\psi_1$ and $\psi_2$, making $\psi$ a minimal totally complex surface in $\mathbb{H}P^{n-1}$.
\end{proof}
\section{Proof of Main Theorem 2A and Main Theorem 2B}\label{section:Dim3}
In this section we give proofs of \ref{MT2A} and \ref{MT2B}. We will consider 3-dimensional minimal Lagrangian submanifolds $M^3\subset \mathbb{C}^3$, with $j$ the canonical structure on $\mathbb{C}^3$.  Lemma \ref{lem:2ndfund} and Lemma \ref{lem:connect} then imply that there is a local frame $\left\{E_1, E_2, E_3 \right\}$ satisfying the following conditions for $i=1,2,3$:
\begin{alignat*}{3}
&h(E_1, E_1)&&=\lambda jE_1, \qquad{} &&h(E_1, E_2)=-\lambda jE_2,\\
&h(E_2, E_2)&&=-\lambda jE_1,\qquad{} &&h(E_i, E_3)=0,
\end{alignat*}
with connection coefficients defined by
\begin{alignat*}{3}
\nabla_{E_1}E_1&=\gamma_{11}^2E_2+\gamma_{11}^3E_3, \qquad{}&\nabla_{E_1}E_2&=-\gamma_{11}^2E_1-\gamma_{21}^3E_3,\qquad{}&\nabla_{E_1}E_3&=-\gamma_{11}^3E_1+\gamma_{21}^3E_2,\\
\nabla_{E_2}E_1&=\gamma_{21}^2E_2+\gamma_{21}^3E_3, &\nabla_{E_2}E_2&=-\gamma_{21}^2E_1+\gamma_{11}^3E_3,&\nabla_{E_2}E_3&=-\gamma_{21}^3E_1+\gamma_{11}^3E_2,\\
\nabla_{E_3}E_1&=\frac{\gamma_{21}^3}{3} E_2, &\nabla_{E_3}E_2&=-\frac{\gamma_{21}^3}{3}E_1,&\nabla_{E_3}E_3&=0.
\end{alignat*}
The function $\lambda$ also satisfies the following system of differential equations: 
\begin{align}
E_1(\lambda)&=-3\gamma_{21}^2\lambda,\label{eq:n3L1}\\
E_2(\lambda)&=3\gamma_{11}^2\lambda,\label{eq:n3L2}\\
E_3(\lambda)&=\gamma_{11}^3\lambda.\label{eq:n3L3}
\end{align}
\subsection{Proof of Main Theorem 2A}
We are now able to present the proof of the following proposition, which is the first part of \ref{MT2A}.
\begin{proposition}\label{prop:P1T2A}
Let $M^3\subset \mathbb{C}^3$ be a minimal $\delta(2)$-ideal Lagrangian submanifold, with an integrable but not totally geodesic distribution $\mathcal{D}_1$ as in (\ref{eq:IntroD1}). The image of $M^3$ under the map $\psi$, defined in \ref{MT1}, is a minimal totally complex surface in $\mathbb{H}P^2$, which is moreover anti-symmetric.
\end{proposition}
\begin{proof}
Let the distribution $\mathcal{D}_1$ of the minimal $\delta(2)$-ideal submanifold $M^3$, as defined in (\ref{eq:IntroD1}), be integrable, but not totally geodesic. This implies that the function $\gamma_{21}^3$, introduced in the beginning of this section, is everywhere zero. \ref{MT1} then implies that the image of $M^3$ under the map $\psi$ is a minimal totally complex surface in $\mathbb{H}P^2$. The only thing left to prove is then that this surface is also anti-symmetric. 

We begin the proof by defining local coordinates $(u,v,t)$ on the the minimal $\delta(2)$-ideal submanifold $M^3$, in a similar way as in \cite{Affinesfeer}. Consider a function $\alpha$ on $M^3$ defined by
\begin{align*}
\alpha=\sqrt[3]{\frac{1}{\lambda (\gamma_{11}^3)^2}}.
\end{align*}
Note that $\alpha$ is well-defined since the distribution $\mathcal{D}_1$ is not totally geodesic, implying that the function $\gamma_{11}^3$ is nowhere zero. As $\mathbb{C}^3$ is flat, the Gauss equation (\ref{Gaussequation}) reduces to $~{ R(X,Y)Z=A_{h(Y,Z)}X - A_{h(X,Z)}Y}$.
Computing al the components of the Gauss equation with respect to the frame $\left\{E_1, E_2, E_3 \right\}$ then yields
\begin{align}
E_1(\gamma_{21}^2)-E_2(\gamma_{11}^2)&=2\lambda^2-(\gamma_{11}^2)^2-(\gamma_{21}^2)^2-(\gamma_{11}^3)^2,\label{eq:diff1}\\
E_1(\gamma_{11}^3)&=0,\\
E_2(\gamma_{11}^3)&=0,\\
E_3(\gamma_{11}^2)&=\gamma_{11}^2\gamma_{11}^3,\\
E_3(\gamma_{21}^2)&=\gamma_{21}^2\gamma_{11}^3,\\
E_3(\gamma_{11}^3)&=(\gamma_{11}^3)^2.
\end{align}
The compatibility condition for this system of equations yields the additional equation
\begin{align}
E_1(\gamma_{11}^2)+E_2(\gamma_{21}^2)=0.\label{eq:diff2}
\end{align}
It is now straightforward to check that the vector fields defined by
\begin{align}
T&=E_3,\label{coT}\\
U&=\alpha E_1+\alpha E_2,\label{coU}\\
V&=-\alpha E_1+\alpha E_2\label{coV},
\end{align}
satisfy $[T,U]=[T,V]=[U,V]=0$. Thus there exist local coordinates $(u,v,t)$ such that $~{\partial_u=U}$, $\partial_v=V$ and $\partial_t=T$. Using the differential equations (\ref{eq:diff1})-(\ref{eq:diff2})  it is then possible to find the following expressions:
\begin{align*}
\gamma_{11}^3&=-\frac{1}{t},\;
\lambda=\frac{e^F}{t},\;
\alpha=te^{-\frac{1}{3}F},
\end{align*} 
where $F$ is a function depending on the coordinates $(u,v)$ and satisfying the differential equation 
\begin{align}
\Delta F=(3-6e^{2F})e^{-\frac{2}{3}F}\label{eq:LaplaceDI}, 
\end{align} 
where $\Delta$ is the Euclidean Laplacian. Solving the differential equations (\ref{eq:n3L1})-(\ref{eq:n3L3}) for the function $\lambda$ finally yields expressions for the functions $\gamma_{11}^2$ and $\gamma_{21}^2$. Thus the function $F$ completely characterizes the submanifold $M^3$, as every introduced function that determines $M^3$ can be described using $F$. These functions and the vector fields $U, V$ and $T$, as defined in (\ref{coT})-(\ref{coV}) will be used to find a correspondence between the submanifold $M^3$ and a surface in $\mathbb{H}P^{2}$.

\ref{MT1} implies that the image of the map $~{\psi:M^3\rightarrow \mathbb{H}P^2: p\mapsto [E_1(p)+iE_2(p)]}$ is a minimal totally complex surface, as the distribution $\mathcal{D}_1$ is non-totally geodesic. Recall that $\psi$ is the composition of the Hopf fibration $\pi$ together with the map
\begin{align}
x:\mathbb{C}^3\rightarrow S^{11}:\frac{1}{\sqrt{2}}(E_1(p)+iE_2(p)).\label{mapx}
\end{align}
As the vector field $T$ is equal to the vector field $E_3$, we know that $(\de \psi)(T)=0$ from the proof of \ref{MT1}. One then only has to consider the tangent vector fields $x_u$ and $x_v$ of the map $x$ with respect to the coordinate vector fields $\partial_u=U$ and $\partial_v=V$,
\begin{align*}
x_u&=\frac{\alpha}{\sqrt{2}}\left(\gamma_{11}^3(1+i)E_3 - i(\gamma_{11}^2+\gamma_{21}^2)(E_1+iE_2)+\lambda j(E_1+iE_2)-\lambda k(E_1+iE_2)\right),\\
x_v&=\frac{\alpha}{\sqrt{2}}\left(\gamma_{11}^3(i-1)E_3 + i(\gamma_{11}^2-\gamma_{21}^2)(E_1+iE_2)-\lambda j(E_1+iE_2)-\lambda k(E_1+iE_2)\right).
\end{align*}
The horizontal part of these vector fields are then given by
\begin{align*}
\chi_u&=\frac{\alpha}{\sqrt{2}}\gamma_{11}^3(1+i)E_3,\\
\chi_v&=\frac{\alpha}{\sqrt{2}}\gamma_{11}^3(i-1)E_3.
\end{align*}
As in the proof of \ref{MT1}, we will denote by $\psi_u=(\de \pi)(\chi_u)$ and$ \psi_v=(\de \pi)(\chi_v)$ the tangent vectors to the surface $\psi(M^3)$. A calculation shows that the vector fields have length 
\begin{align*}
g(\psi_u,\psi_u)&=g(\psi_v,\psi_v)=\frac{1}{2}e^{-\frac{2}{3}F},\;\;\;g(\psi_u,\psi_v)=0,
\end{align*}
where $g$ is the induced metric tensor on $\mathbb{H}P^2$, making $(u,v)$ isothermal coordinates on the surface $\psi(M^3)$.

The calculation of the second fundamental form $h$ of the surface $\psi(M^3)$ with respect to the coordinates $(u, v)$ becomes more involved since the tangent vectors $x_u$ and $x_v$ of the map $x$ are not the horizontal lifts of $\psi_u$ and $\psi_v$ respectively. This can be solved by utilising a parametrisation of $S^3$ such that the horizontal vector fields $\chi_u$ and $ \chi_v$ can be written in terms of $x_u$ and $x_v$. More specifically, we will consider the maps $\tilde{x}:S^3\times M^3\rightarrow S^{11}:(y,p)\mapsto \frac{1}{\sqrt{2}}y(E_1(p)+iE_2(p))$ and $\tilde{\psi}$ as the composition of $\pi$ with $\tilde{x}$. One can write an arbitrary element $y\in S^3$ as $y_1+iy_2+jy_3+ky_4$, with $y_1^2+y_2^2+y_3^2+y_4^2=1$ and $y_1, y_2, y_3,y_4\in \mathbb{R}$. Consider the following vector fields on $S^3$,
\begin{align*}
X_1(y_1+iy_2+jy_3+ky_4)&=(y_1+iy_2+jy_3+ky_4)\cdot i,\\
X_2(y_1+iy_2+jy_3+ky_4)&=(y_1+iy_2+jy_3+ky_4)\cdot j,\\
X_3(y_1+iy_2+jy_3+ky_4)&=(y_1+iy_2+jy_3+ky_4)\cdot k.
\end{align*}
The tangent vector fields $\tilde{x}_u,\tilde{x}_v$ are then simply 
\begin{align*}
\tilde{x}_u&=(y_1+iy_2+jy_3+ky_4)x_u,\\
\tilde{x}_v&=(y_1+iy_2+jy_3+ky_4)x_v.
\end{align*}
In an analogous way we also obtain expressions for the horizontal vector fields $\tilde{\chi}_u$ and $\tilde{\chi}_v$. Using the vector fields $X_1, X_2, X_3$ it is possible to write 
\begin{align*}
\tilde{\chi}_u&=\tilde{x}_u+\frac{\alpha}{\sqrt{2}}\left((\gamma_{11}^2+\gamma_{21}^2)\de\tilde{x}(X_1)-\lambda(\de\tilde{x})(X_2)+\lambda\de\tilde{x}(X_3)\right),\\
\tilde{\chi}_v&=\tilde{x}_v+\frac{\alpha}{\sqrt{2}}\left((\gamma_{21}^2-\gamma_{11}^2)\de\tilde{x}(X_1)+\lambda(\de\tilde{x})(X_2)+\lambda\de\tilde{x}(X_3)\right).
\end{align*}
These expressions will now enable us to calculate the second fundamental form $h$ of the surface $\psi(M^3)$ in $\mathbb{H}P^2$. Consider the type (1, 2) tensor $\tilde{\mathcal{H}}$ associated to the map $\tilde{x}$, defined as
\begin{equation}\label{eq:H}
\begin{split}
\tilde{\mathcal{H}}(\partial_u, \partial_u)&=\nor\left(\hor\left(D_{\left(\partial_u+\frac{\alpha}{\sqrt{2}}\left((\gamma_{11}^2+\gamma_{21}^2)X_1-\lambda X_2+\lambda X_3\right)\right)}\tilde{\chi}_u\right)\right),\\
\tilde{\mathcal{H}}(\partial_u, \partial_v)&=\nor\left(\hor\left(D_{\left(\partial_u+\frac{\alpha}{\sqrt{2}}\left((\gamma_{11}^2+\gamma_{21}^2)X_1-\lambda X_2+\lambda X_3\right)\right)}\tilde{\chi}_v\right)\right),\\
\tilde{\mathcal{H}}(\partial_v, \partial_u)&=\nor\left(\hor\left(D_{\left(\partial_u+\frac{\alpha}{\sqrt{2}}\left((\gamma_{21}^2-\gamma_{11}^2)X_1+\lambda X_2+\lambda X_3\right)\right)}\tilde{\chi}_u\right)\right),\\
\tilde{\mathcal{H}}(\partial_v, \partial_v)&=\nor\left(\hor\left(D_{\left(\partial_u+\frac{\alpha}{\sqrt{2}}\left((\gamma_{21}^2-\gamma_{11}^2)X_1+\lambda X_2+\lambda X_3\right)\right)}\tilde{\chi}_v\right)\right).
\end{split}
\end{equation}
One can then straightforwardly identify $\tilde{\mathcal{H}}$ with a type (1, 2) tensor $\mathcal{H}$ associated to the map $x$.
Denote with $\sigma(X,Y)$ the part of $\mathcal{H}(X,Y)$ lying in the direction of $(E_1-E_2)$ and $i(E_1-iE_2)$ and with $\mathcal{H}$ the part lying in the direction of $jE_3$ and $kE_3$. Then we find for the tensor $\mathcal{H}$:
\begin{equation}\label{eq:hphi}
\begin{split}
\mathcal{H}(\partial_u, \partial_u)&=\sigma(\partial_u, \partial_u)+\kappa(\partial_u, \partial_u),\\
\mathcal{H}(\partial_v, \partial_v)&=\sigma(\partial_v, \partial_v)+\kappa(\partial_v, \partial_v),\\
\mathcal{H}(\partial_u, \partial_v)&=\sigma(\partial_u, \partial_v)+\kappa(\partial_u, \partial_v).
\end{split}
\end{equation}
The tensors $\sigma$ and $\kappa$ satisfy the following relations:
\begin{equation}\label{eq:sigmakappa}
\begin{split}
\sigma(\partial_v, \partial_v)&=-\sigma(\partial_u, \partial_u),\quad \kappa(\partial_v,\partial_v)=-\kappa(\partial_u, \partial_u),\\
\sigma(\partial_u,\partial_v)&=\sigma(\partial_v, \partial_u),\quad\;\;\, \kappa(\partial_u,\partial_v)=\kappa(\partial_v,\partial_u),\\
\sigma(\partial_u,\partial_v)&=i\sigma(\partial_u,\partial_u),\quad\,  \kappa(\partial_u,\partial_v)=-i \kappa(\partial_u,\partial_u),
\end{split}
\end{equation}
with explicit formulations
\begin{align*}
\sigma(\partial_u, \partial_u)&=-\frac{\alpha^2(\gamma_{11}^3)^2}{\sqrt{2}}i(E_1-iE_2),\\
\kappa(\partial_u, \partial_u)&=\lambda\alpha^2\gamma_{11}^3kE_3.
\end{align*}
Note that, due to the defining equations (\ref{eq:H}) of the tensor $\tilde{\mathcal{H}}$, one can find the second fundamental form $h$ of the surface $\psi(M^3)$ by applying the Hopf-fibration to the tensor $\mathcal{H}$. The observation that the tensors $\sigma$ and $\kappa$ have to satisfy the previous equations, together with Definition \ref{def:anti-symm}, imply that the surface $\psi(M^3)$ is a minimal anti-symmetric totally complex surface, which is what we had to prove.
\end{proof}

The following proposition shows a characterization for minimal anti-symmetric totally complex surfaces in the quaternionic projective space $\mathbb{H}P^2$, which will be used in the proof of \ref{MT2A}.
\begin{proposition}\label{prop:charN2}
Let $N^2$ be an immersed minimal anti-symmetric totally complex surface of the quaternionic projective space  $\mathbb{H}P^2$. Then the metric $g$ and second fundamental form $h$ of the immersion solely depend on a function $f: N^2\rightarrow\mathbb{R}$, which satisfies the following relation:
\begin{align}\label{eq:charf}
\Delta_g f&=6(1-2e^{2f}),
\end{align}
where $\Delta_g$ is the Laplace-Beltrami operator. 
\end{proposition}
\begin{proof}
Let $N^2\subset\mathbb{H}P^2$ be an immersed minimal anti-symmetric totally complex surface with local frame $\left\{E_1,E_2\right\}$. Then there exists a local adapted basis $\left\{ I,J,K\right\}$ of the quaternionic structure of $\mathbb{H}P^2$, acting on $TN^2$, such that $E_2=IE_1$ and that $JE_1, KE_1$ are normal to the surface. Note that this fixes the section $I$ on the surface, but not the sections $J$ and $K$. Recall that being anti-symmetric implies that $g( h(X,Y),JZ)$ and $g( h(X,Y),KZ )$ are anti-symmetric in the last two indices, with $g$ the metric tensor of $\mathbb{H}P^2$,  $h$ the second fundamental form of $N^2$ and $X,Y,Z$ vector fields on $N^2$. It is then possible, since $N^2$ is also minimal,  to fix the section $J$ of the local adapted basis by writing the second fundamental form as 
\begin{align}
h(E_1,E_1)&=\xi+\beta JE_1,\label{eq: 2ndfundform1}\\
h(E_1,E_2)&= I\xi-\beta KE_1,\label{eq: 2ndfundform2}\\
h(E_2,E_2)&=- \xi-\beta JE_1,\label{eq: 2ndfundform3}
\end{align}
where $\xi$ is a unit vector field orthogonal to $E_1, IE_1, JE_1, KE_1$ and $\beta$ are smooth functions on $N^2$ and $\beta$ also a non-zero function. Denote the Levi-Civita connections of $N^2$ and $\mathbb{H}P^2$ by $\nabla$ and $\tilde{\nabla}$ respectively. We consider smooth functions $a,b$ on $N^2$ such that the connection $\nabla$ is defined as 
\begin{align*}
\nabla_{E_1}E_1&=aE_2,\quad \nabla_{E_1}E_2=-aE_1,\\
\nabla_{E_2}E_1&=-bE_2,\quad \nabla_{E_2}E_2=bE_1.
\end{align*}
Since $\left\{ I,J,K\right\}$ is a local adapted basis of the quaternionic structure of $\mathbb{H}P^2$ defined on $TN^2$, there are 1-forms $r,p$ and $q$  defined on a neighbourhood $U$ of $\mathbb{H}P^2$ around a point $p$ of $N^2$ such that for every tangent vector field $X$ on $N^2$  one has
\begin{align*}
\tilde{\nabla}_X I&=r(X)J-q(X)K,\\
\tilde{\nabla}_X J&=-r(X)I+p(X)K,\\
\tilde{\nabla}_X K&=q(X)I-p(X)J.
\end{align*} 
Using the formula of Gauss (\ref{Gaussformula}) for tangent vector fields $X$ and $Y$ on $N^2$, we can write 
\begin{align*}
(\tilde{\nabla}_X I)Y=\nabla_X(IY)+h(X,IY)-I\nabla_X Y-Ih(X,Y).
\end{align*}
One can reformulate this as $~{h(X,IY)-Ih(X,Y)=r(X)JY-q(X)KY}$, as $N^2$ is a totally complex surface with respect to $I$. This results, together with equations (\ref{eq: 2ndfundform1})-(\ref{eq: 2ndfundform3}), in the following relations:
\begin{alignat*}{3}
&r(E_1)&&=0,\quad &&r(E_2)=-2\beta,\\
&q(E_1)&&=2\beta,\quad &&q(E_1)=0. 
\end{alignat*}
Using these equalities and denoting $p(E_1)$ and $p(E_2)$ as $p_1$ and $p_2$ respectively, the connections $\tilde{\nabla}_{E_1}$ and $\tilde{\nabla}_{E_2}$ can be written as 
\begin{align*}
&\tilde{\nabla}_{E_1}E_1=aE_2+\beta JE_1+\xi,\\
&\tilde{\nabla}_{E_1}E_2=-aE_1-\beta KE_1+I\xi,\\
&\tilde{\nabla}_{E_1}JE_1=-\beta E_1+(p_1-a)KE_1+J\xi,\\
&\tilde{\nabla}_{E_1}KE_1=\beta E_2+(a-p_1)JE_1+\alpha K\xi,\\
&\tilde{\nabla}_{E_1}\xi=-E_1+c_1I\xi+c_2J\xi+c_3K\xi,\\
&\tilde{\nabla}_{E_1}I\xi=-E_2-c_1\xi-c_3J\xi+c_2\beta K\xi,\\
&\tilde{\nabla}_{E_1}J\xi=-JE_1-c_2\xi+c_3I\xi+(p_1-c_1)K\xi,\\
&\tilde{\nabla}_{E_1}K\xi=-KE_1-c_3\xi+(2\beta-c_2)I\xi+(c_1-p_1)J\xi,\\
&\tilde{\nabla}_{E_2}E_1=-bE_2-\beta KE_1+I\xi,\\
&\tilde{\nabla}_{E_2}E_2=bE_1-\beta JE_1-\xi,\\
&\tilde{\nabla}_{E_2}JE_1=\beta E_2+(b+p_2)KE_1-\xi,\\
&\tilde{\nabla}_{E_2}KE_1=\beta E_1-(b+p_2)KE_1+J\xi,\\
&\tilde{\nabla}_{E_2}\xi=E_2+d_1 I\xi+d_2 J\xi+d_3K\xi,\\
&\tilde{\nabla}_{E_2}I\xi=-E_1-d_1\xi-(d_3+2\beta)J\xi+d_2\xi,\\
&\tilde{\nabla}_{E_2}J\xi=-KE_1-d_2\xi+(d_3+2\beta)I\xi+(p_2-d_1)K\xi,\\
&\tilde{\nabla}_{E_2}K\xi=JE_1-d_3\xi-d_2 I\xi-(p_2-d_1)J\xi,
\end{align*}
where $c_1, c_2, c_3,d_1,d_2$ and $d_3$ are smooth functions defined on $N^2$.
Computing the Codazzi equation (\ref{Codazzi}), by substituting $(X,Y,Z)$ by $(E_1, E_2, E_1)$, results in the algebraic relations
\begin{alignat*}{3}
&d_2&&=-c_3,\qquad{}&&d_3=c_2-2\beta,\\
&c_1&&=2a,\qquad{}&&d_1=-2b,
\end{alignat*}
while also yielding differential equations for the function $\beta$,
\begin{align*}
E_1(\beta)&=(b-p_2)\beta,\\ 
E_2(\beta)&=(a+p_1)\beta.
\end{align*}
By substituting $(X,Y,Z)$ by $(E_1, E_2, E_1)$ in the Gauss equation (\ref{Gaussequation}) one also obtains a differential equation for the functions $a$ and $b$: 
\begin{align}
E_2(a)+E_1(b)&=2+a^2+b^2-2\beta^2.\label{eq:Gaussbeta}
\end{align}
Note that there is still a degree of freedom regarding our choice of local frame $\left\{E_1, E_2 \right\}$ as it is still possible to rotate the vector field $E_1$ in the plane spanned by $\left\{E_1, E_2 \right\}$. To simplify notation we will write $e^{i\theta}X=\cos{\theta}X+\sin(\theta)IX$, for a vector field $X$ on $N^2$. Since $E_2=IE_1$, one gets that by rotating $E_1$ over an angle $\theta$, equation (\ref{eq: 2ndfundform1}) becomes
\begin{align*}
h(e^{i\theta}E_1, e^{i\theta}E_1)=e^{2i\theta}\xi+\beta e^{-i\theta}J(e^{i\theta}E_1).
\end{align*}
Thus the vector field $\xi$ has to be rotated over an angle $2\theta$ if $E_1$ is rotated over an angle $\theta$ and $J$ has to be redefined as $e^{-i\theta}J$, or written explicitly 
\begin{align*}
E_1'=e^{i\theta}E_1,\;\;\xi'=e^{2i\theta}\xi,\;\;J'=e^{-i\theta}J.
\end{align*}
By comparing $\tilde{\nabla}_{E_1}\xi$ with $\tilde{\nabla}_{E_1'}\xi'$ one can see that by rotating over an appropriate angle $\theta$, the vector field $\tilde{\nabla}_{E_1'}\xi'$ has no part lying in the direction of $K\xi$. Thus we can, without loss of generality, take $c_3$ to be zero. We introduce the function $\alpha=c_2-\beta$ to simplify the calculations. Applying the Ricci equation (\ref{Ricci}) and substituting $(X,Y,\zeta, \eta$) by $(E_1,E_2,\xi, I\xi), (E_1,E_2,\xi, J\xi)$ and $(E_1,E_2,\xi, K\xi)$, while using that $c_3=0$, we deduce that
\begin{align}
E_2(a)+E_1(b)&=2+a^2+b^2-\beta^2-\alpha^2,\label{eq:Gaussbeta2}\\
E_2(\alpha)&=\alpha(5a-p_1),\\
E_1(\alpha)&=\alpha(5b+p_2).\label{eq:Gaussalpha}
\end{align}
The compatibility condition for this system of equations yields the additional equation $E_1(a)-E_2(b)=0$. Comparing equations (\ref{eq:Gaussbeta}) and (\ref{eq:Gaussbeta2}) then yields that $\alpha^2=\beta^2$ and without loss of generality we can take $\alpha=\beta$. This also implies that $p_1=2a$, $p_2=-2b$ and $c_2=2\beta$. We thus have have the following system of differential equations
\begin{align}
E_1(\beta)&=3b\beta,\label{eq:diffbeta1}\\
E_2(\beta)&=3a\beta,\label{eq:diffbeta2}\\
E_1(a)-E_2(b)&=0,\label{eq:diffa1b2}\\
E_2(a)+E_1(b)&=2+a^2+b^2-2\beta^2\label{eq:diffa2b1}.
\end{align}
Consider a smooth function $\rho$ on $N^2$ defined by $\rho=(2\beta)^{-1/3}$ and construct vector fields $U$ and $V$ on $N^2$ as
\begin{align*}
U&=\rho E_1,\;V=\rho E_2.
\end{align*}
One can prove using equations (\ref{eq:diffbeta1})-(\ref{eq:diffbeta2}) that $[U,V]=0$, thus we can define coordinates $(u,v)$ such that $\partial_u=U$ and $\partial_v=V$. In particular, these coordinates are isothermal coordinates with corresponding metric $g= (2\beta)^{-2/3}(\de u^2+\de v^2)$. Using the coordinate vector fields $\partial_u$ and $\partial_v$, it is possible to show that equation (\ref{eq:diffa1b2}) is trivially satisfied and that equation (\ref{eq:diffa2b1}) implies that the function  $\beta$ satisfies the following differential equation:
\begin{align*}
\Delta(\log \beta)=-\frac{3\sqrt[3]{2}(\beta^2-1)}{\beta^{2/3}},
\end{align*}
where $\Delta$ is the Euclidean Laplacian. Note that if one defines a new function $f$ on $N^2$ as $\beta=\sqrt{2}e^{f}$, the previous equation reduces to $\Delta f=(3-6e^{2f})e^{-\frac{2}{3}f}$, with the metric given by 
\begin{align*}
g= \frac{1}{2}e^{-\frac{2}{3}f}(\de u^2+\de v^2).
\end{align*}
Recall that the Laplace-Beltrami operator $\Delta_g$ of a metric $~{g=\phi(u,v)(\de u^2+\de v^2)}$ with respect to isothermal coordinates $(u,v)$ is given by 
\begin{align}\label{eq:metricLaplacian}
\Delta_g=\frac{1}{\phi(u,v)}\left(\frac{\partial^2}{\partial u^2}+\frac{\partial^2}{\partial v^2} \right).
\end{align} 
One can thus the previous differential equation as $\Delta_g f=6(1-2e^{2f})$,
which is exactly equation (\ref{eq:charf}) from the proposition. The function $f$, defined by the previous equation, then completely determines the surface $N^2$.
\end{proof}
We are now able to formulate the proof of \ref{MT2A}.
\begin{proof}
Note that the first part of the theorem has already been proven in Proposition \ref{prop:P1T2A}. Thus, all that is left to prove is the second part of the theorem. Take $N^2$ to be a minimal anti-symmetric totally complex surface in $\mathbb{H}P^2$. Proposition \ref{prop:charN2} then states that $N^2$ can be characterized by a non-zero function $f$ satisfying differential equation (\ref{eq:charf}). From the proof of Proposition \ref{prop:charN2} and the differential equation (\ref{eq:charf}) of the function $f$, we have that the metric $g$ with respect to the isothermal coordinates $(u,v)$ is given by $g= \frac{1}{2}e^{-\frac{2}{3}f}(\de u^2+\de v^2)$. Note that the differential equation (\ref{eq:charf}) is then exactly the same as equation (\ref{eq:LaplaceDI}), which characterizes a minimal $\delta(2)$-ideal Lagrangian submanifold in $\mathbb{\mathbb{C}}^3$, with distribution $\mathcal{D}_1$ being integrable but not totally geodesic. 

Let $S$ be an open domain in $\mathbb{R}^2$ and $f$ a function satisfying the differential equation (\ref{eq:LaplaceDI}) on $S$, characterizing the surface $N^2$. We can then define functions $\gamma_{11}^2, \gamma_{21}^2, \gamma_{11}^3, \lambda,\alpha$ as described in the proof of Proposition \ref{prop:P1T2A}, by identifying the function $f$ with the function $F$. It is now possible to construct vector fields $E_1,E_2$ and $E_3$ on $S\times\mathbb{R}$ by
\begin{align*}
E_1&=\frac{1}{2\alpha}\partial_u-\frac{1}{2\alpha}\partial_v,\\
E_2&=\frac{1}{2\alpha}\partial_u+\frac{1}{2\alpha}\partial_v,\\
E_3&=\partial_t,
\end{align*}
where $\partial_u, \partial_v$ are the coordinate vector fields on $S$ and $\partial_t$ is the coordinate vector field on $\mathbb{R}$. We can introduce a metric tensor $g$ on $S\times\mathbb{R}$ by imposing that the vector fields $E_1,E_2$ and$E_3$ form an orthonormal basis. By defining a symmetric $(1,2)$-tensor field $h$ on $S\times\mathbb{R}$ by
\begin{alignat*}{4}
&h(E_1, E_1)&&=\lambda jE_1, \qquad{} &&h(E_1, E_2)=-\lambda jE_2, \qquad{}&&h(E_1,E_3)=0,\\
&h(E_2, E_2)&&=-\lambda jE_1,\qquad{} &&h(E_2, E_3)=0,\qquad{} &&h(E_3,E_3)=0, 
\end{alignat*}
one can immerse $S\times\mathbb{R}$ in $\mathbb{C}^3$, with $j$ denoting the canonical complex structure on $\mathbb{C}^3$, by taking $h$ as the second fundamental form of $S\times\mathbb{R}$. It is finally a straightforward calculation to show that the distribution $\mathcal{D}_1=\spanned\left\{E_1, E_2 \right\}$ is integrable and not totally geodesic and that $S\times\mathbb{R}$ is a minimal $\delta(2)$-ideal Lagrangian immersion of $\mathbb{C}^3$.
\end{proof}
\subsection{Proof of Main Theorem 2B}
We will now consider the case in which the distribution $\mathcal{D}_1$ is nowhere integrable and present the proof of \ref{MT2B}. One can first prove the following proposition, which is an analogous version of Proposition \ref{prop:P1T2A}.
\begin{proposition}\label{prop:P1T2B}
Let $M^3\subset \mathbb{C}^3$ be a minimal $\delta(2)$-ideal Lagrangian submanifold, with nowhere integrable distribution $\mathcal{D}_1$ as in (\ref{eq:IntroD1}). The image of $M^3$ under the map $\psi$, defined in \ref{MT1}, is a minimal totally complex surface in $\mathbb{H}P^2$, which is moreover anti-symmetric.
\end{proposition}
\begin{proof}
As in the proof of Proposition \ref{prop:P1T2A}, we begin by defining local coordinates $(u,v,t)$ on the the minimal $\delta(2)$-ideal submanifold $M^3$, which will then be used to study the map $\psi$, as defined in \ref{MT1}. In a similar way as in \cite{Affinesfeer}, we define functions $w, \alpha_1, \alpha_2, \beta_1$ and $\beta_2$ on $M^3$ by
\begin{align*}
w&=\frac{\gamma_{11}^3}{(\gamma_{11}^3)^2+(\gamma_{21}^3)^2},\\
(\alpha_1+j\alpha_2)^3&=\frac{1}{\lambda(\gamma_{11}^3-j\gamma_{21}^3)^2},\\
\beta_1&=E_1(w),\\
\beta_2&=E_2(w).
\end{align*}
Remark that the functions $\alpha_1, \alpha_2$ are well-defined up to multiplication by $e^{\frac{2}{3}\pi j}$. Computing all the Gauss equations with respect to the frame $\left\{E_1, E_2, E_3 \right\}$, together with the fact that the Levi-Civita connection is torsion-less, yields analogous equations to equations (\ref{eq:diff1})-(\ref{eq:diff2}) found in Proposition \ref{prop:P1T2A}. These can then be used to show that the vector fields defined by 
\begin{align*}
T&=E_3,\\
U&=\alpha_1 E_1+\alpha_2E_2+(\alpha_1\beta_1+\alpha_2\beta_2)E_3,\\
V&=-\alpha_2 E_1+\alpha_1 E_2+(-\alpha_2\beta_1+\alpha_1\beta_2)E_3,
\end{align*}
satisfy $[T,U]=[T,V]=[U,V]=0$. Thus there exist local coordinates $(u,v,t)$ such that $~{\partial_u=U}$, $~{\partial_v=V}$ and $~{\partial_t=T}$. Similar to the proof of Proposition \ref{prop:P1T2A}, it is possible to find the following expressions:
\begin{align*}
\gamma_{11}^3&=\frac{-t}{t^2+C^2}, \\
\gamma_{21}^3&=\frac{-C}{t^2+C^2},\\
\lambda&=\frac{e^F}{\sqrt{t^2+C^2}},\\
\alpha_1^2&=-\alpha_2^2+e^{-\frac{2}{3}F}(t^2+C^2),
\end{align*}
with $C$ and $F$ functions on $M^3$, depending on the coordinates $(u,v)$ and satisfying the following system of differential equations
\begin{align}
\Delta C&=-2Ce^{-\frac{2}{3}F},\label{eq:LaplaceC}\\
\Delta F&=(3-6e^{F})e^{-\frac{2}{3}F}\label{eq:LaplaceD}.
\end{align}
Solving the differential equations for $\lambda,\beta_1$ and $\beta_2$ then yields expressions for the functions $\gamma_{11}^2,\gamma_{21}^2, \beta_1$ and $\beta_2$. 

Consider once again the map $\psi:M^3\rightarrow \mathbb{H}P^2$, as defined in \ref{MT1}. Because the distribution $\mathcal{D}_1$ is nowhere integrable, $\psi$ will yield a surface in $\mathbb{H}P^2$ due to \ref{MT1}, and one can show that the tangent vector fields $\psi_u$ and $\psi_v$ form a basis for the tangent bundle of this surface. More specifically, a calculation analogous to one in the proof of Proposition \ref{prop:P1T2A} shows that
\begin{align*}
g(\psi_u,\psi_u)&=g(\psi_v,\psi_v)=\frac{1}{2}e^{-\frac{2}{3}F},\;\;\;g(\psi_u,\psi_v)=0,
\end{align*}
where $g$ is the induced metric tensor on $\mathbb{H}P^2$, making $(u,v)$ once again isothermal coordinates on the surface $\psi(M^3)$. In a completely analogous manner as in the proof of Proposition \ref{prop:P1T2A}, one can also associate to the map $x:M^3\rightarrow S^{11}$, as defined in equation (\ref{mapx}), a (1,2) type tensor $\mathcal{H}$ with the second fundamental form $h$ of the surface $\psi(M^3)$ given by $\dd\pi(\mathcal{H)}$. Denote once again with $\sigma(X,Y)$ the part of $\mathcal{H}(X,Y)$ lying in the direction of $(E_1-E_2)$ and $i(E_1-iE_2)$ and with $\kappa(X,Y)$ the part lying in the direction of $jE_3$ and $kE_3$. One can then show that the tensors $\mathcal{H}, \sigma$ and $\kappa$ also satisfy the relations shown in equations (\ref{eq:H}) and (\ref{eq:hphi}), which were obtained when the distribution $\mathcal{D}_1$ is integrable but not totally geodesic. Note that even though the tensors $\sigma$ and $\kappa$ follow the same relations as before, their explicit formulation is more complicated: 
\begin{align*}
\sigma(\partial_u, \partial_u)&=\frac{1}{\sqrt{8}}\left(((\alpha_1\gamma_{12}^3-\alpha_2\gamma_{11}^3)^2-(\alpha_1\gamma_{11}^3+\alpha_2\gamma_{21}^3)^2)(E_1-iE_2)\right)\\ &\quad{}+\frac{1}{\sqrt{2}}(\alpha_1\gamma_{11}^3+\alpha_2\gamma_{21}^3)(\alpha_1\gamma_{21}^3-\alpha_2\gamma_{11}^3)i(E_1-iE_2),\\
\kappa(\partial_u, \partial_u)&=\lambda(\gamma_{11}^3(\alpha_2^2-\alpha_1^2)-2\alpha_1\alpha_2\gamma_{21}^3)jE_3
+\lambda(\gamma_{21}^3(\alpha_2^2-\alpha_1^2)+2\alpha_1\alpha_2\gamma_{11}^3)kE_3.
\end{align*}
Thus Definition \ref{def:anti-symm} then implies that $\psi(M^3)$ is once again minimal totally complex surface, while also being anti-symmetric.
\end{proof}
We are now able to formulate the proof of \ref{MT2B}.
\begin{proof}
The first part of this theorem has already been proven in Proposition \ref{prop:P1T2B}, meaning that we only need to prove the second part of this theorem. Let $N^2$ to be a minimal anti-symmetric totally complex surface in $\mathbb{H}P^2$. Proposition \ref{prop:charN2} then shows that $N^2$ can be characterized by a function $f$ satisfying differential equation (\ref{eq:charf}). The proof of \ref{MT2A} then shows that this equation is equivalent with equation (\ref{eq:LaplaceD}), being one of the equations characterizing a minimal $\delta(2)$-ideal Lagrangian submanifold in $\mathbb{C}^3$, with distribution $\mathcal{D}_1$ being nowhere integrable. Take $\tilde{f}$ to be an eigenfunction of the Laplace-Beltrami operator $\Delta_g$ with eigenvalue $-4$, meaning that $\Delta_g \tilde{f}=-4\tilde{f}$. This expression can be rewritten as $\Delta \tilde{f}=-2\tilde{f}e^{-\frac{2}{3}f}$, with $\Delta$ the usual Euclidean Laplacian. Note that this is exactly the same as equation (\ref{eq:LaplaceC}), being the other equation in the set of differential equations to characterize characterize a minimal $\delta(2)$-ideal Lagrangian submanifold in $\mathbb{C}^3$, with distribution $\mathcal{D}_1$ being nowhere integrable. 

Let $S$ then be an open domain in $\mathbb{R}^2$ and $f$ a function satisfying equation (\ref{eq:charf}) en $\tilde{f}$ an eigenfunction of the Laplace-Beltrami-operator with eigenvalue $-4$. One can now define functions $\gamma_{11}^2, \gamma_{21}^2, \gamma_{11}^3, \gamma_{12}^3, \lambda,\alpha_1, \alpha_2, \beta_1$ and $\beta_2$ on $S\times\mathbb{R}$ as described in the proof of Proposition \ref{prop:P1T2B}, by identifying the functions $f$ and $\tilde{f}$ with the functions $D$ and $C$ respectively. It is then possible to construct vector fields $E_1,E_2$ and $E_3$ on $S\times\mathbb{R}$ by
\begin{align*}
E_1&=\frac{\alpha_1}{\alpha_1^2+\alpha_2^2}\partial_u-\frac{\alpha_2}{\alpha_1^2+\alpha_2^2}\partial_v-\beta_1\partial_t,\\
E_2&=\frac{\alpha_2}{\alpha_1^2+\alpha_2^2}\partial_u+\frac{\alpha_1}{\alpha_1^2+\alpha_2^2}\partial_v-\beta_2\partial_t,\\
E_3&=\partial_t.
\end{align*}
Note that $\partial_u, \partial_v$ are the coordinate vector fields on $S$ and $\partial_t$ the coordinate vector field on $\mathbb{R}$. It is then possible to introduce a metric tensor $g$ on $S\times\mathbb{R}$ by imposing that the vector fields $E_1,E_2$ and $E_3$ form an orthonormal basis. By defining a symmetric $(1,2)$-tensor field $h$ on $S\times\mathbb{R}$ by
\begin{alignat*}{4}
&h(E_1, E_1)&&=\lambda jE_1, \qquad{} &&h(E_1, E_2)=-\lambda jE_2, \qquad{}&&h(E_1,E_3)=0,\\
&h(E_2, E_2)&&=-\lambda jE_1,\qquad{} &&h(E_2, E_3)=0,\qquad{} &&h(E_3,E_3)=0, 
\end{alignat*}
one can immerse $S\times\mathbb{R}$ in $\mathbb{C}^3$, with $j$ denoting the canonical complex structure on $\mathbb{C}^3$ by taking $h$ as the second fundamental form of $S\times\mathbb{R}$ in $\mathbb{C}^3$. It is straightforward to check that the distribution $\mathcal{D}_1=\spanned\left\{E_1, E_2 \right\}$ is nowhere integrable and that $S\times\mathbb{R}$ is a minimal $\delta(2)$-ideal Lagrangian immersion of $\mathbb{C}^3$.
\end{proof}

\section{Reformulation of the Main Theorems}\label{Section:reformulation}
In this section a one-to-one correspondence is given between minimal anti-symmetric totally complex surfaces in the quaternionic projective space $\mathbb{H}P^2$ and minimal Lagrangian surfaces in the complex projective space $\mathbb{C}P^2$. Applying this correspondence allows us to reformulate \ref{MT2A} and \ref{MT2B}. We first prove the following proposition.
\begin{proposition}\label{prop:minLagr}
There exists a one-to-one correspondence between minimal anti-symmetric totally complex surfaces in the quaternionic projective space $\mathbb{H}P^2$ and minimal Lagrangian surfaces in the complex projective space $\mathbb{C}P^2$.
\end{proposition}
\begin{proof}
It is possible to characterize an arbitrary minimal Lagrangian surface $L^2\subset \mathbb{C}P^2$ by a function $u$ solving the following differential equation
\begin{align}\label{eq:charu}
2u_{z\bar{z}}=-e^{2u}+|Q|^2e^{-4u}, 
\end{align}
where $z=x+iy$ and $\bar{z}=x-iy$ are complex coordinates on the surface and $Q$ a holomorphic cubic differential as shown in \cite{Lagrangian}. Note that one can choose the complex coordinates $z, \bar{z}$ in such a way that $Q$ can be taken to be any constant. Recall the definition of the Wirtinger derivatives as
\begin{align*}
\partial_z=\frac{1}{2}(\partial_x-i\partial_y),\; \partial_{\bar{z}}=\frac{1}{2}(\partial_x+i\partial_y).
\end{align*}
A straightforward calculation then yields that $\Delta u=4u_{z\bar{z}}$, where $\Delta$ is the usual Euclidean Laplacian.
Proposition \ref{prop:charN2} states that every minimal anti-symmetric totally complex surface $N^2\subset\mathbb{H}P^2$ is determined by a function $f$, which satisfies equation (\ref{eq:charf}). This equation reduces to 
\begin{align*}
2v_{z\bar{z}}=-e^{2v}+4e^{-4v}, 
\end{align*}
where the function $v$ is defined as $v=-\frac{1}{3}f-\frac{\log{2}}{2}$. Note that this is exactly the same differential equation as equation (\ref{eq:charu}), if the complex coordinates $z,\bar{z}$ are taken such that the holomorphic cubic differential $Q$ is normalized as $|Q|^2=4$. One thus has a one-to-one correspondence between minimal Lagrangian surfaces in $\mathbb{C}P^2$ and minimal anti-symmetric totally complex surfaces in $\mathbb{H}P^2$.
\end{proof}
Note that Proposition \ref{prop:minLagr} can be used to rewrite the two main theorems, as it shows that there is a one-to-one correspondence between minimal anti-symmetric totally complex surfaces in $\mathbb{H}P^2$ and minimal Lagrangian surfaces in $\mathbb{C}P^2$. One can then give an equivalent statement of \ref{MT2A}.
\begin{named}{Main Theorem 2A'}\label{MT2A'}
There exists a correspondence between minimal Lagrangian surfaces in $\mathbb{C}P^2$ and minimal $\delta(2)$-ideal Lagrangian submanifolds of $\mathbb{C}^3$, with integrable but not totally geodesic distribution $\mathcal{D}_1$ as defined in equation (\ref{eq:IntroD1}), and minimal Lagrangian surfaces in $\mathbb{C}P^2$.
\end{named}
Analogously one can rewrite \ref{MT2B} in the following manner. 
\begin{named}{Main Theorem 2B'}\label{MT2B'}
There exists a correspondence between minimal Lagrangian surfaces in $\mathbb{C}P^2$, together with an eigenfunction of the Laplace-Beltrami operator with eigenvalue $-4$, and minimal $\delta(2)$-ideal Lagrangian submanifolds of $\mathbb{C}^3$, with integrable but not totally geodesic distribution $\mathcal{D}_1$ as defined in equation (\ref{eq:IntroD1}).
\end{named}
\bibliographystyle{plain}
\bibliography{citationsHP}

\end{document}